\documentclass[11pt]{article}
\usepackage[T1]{fontenc}
\usepackage{times}
\usepackage{rawfonts}
\usepackage{hyperref}
\usepackage{url}
\usepackage{latexsym}
\usepackage{amsmath}
\usepackage{amsfonts}
\usepackage{amssymb}
\usepackage{enumerate}
\usepackage{xspace}
\usepackage{verbatim}
\usepackage{ntheorem}
\theoremseparator{.}
\newtheorem{theorem}{Theorem}[section]
{\theorembodyfont{\rm} \newtheorem{definition}[theorem]{Definition} }
{\theorembodyfont{\rm} \newtheorem{example}[theorem]{Example} }
{\theorembodyfont{\rm} \newtheorem{remark}[theorem]{Remark} }
\newtheorem{lemma}[theorem]{Lemma}
\newtheorem{corollary}[theorem]{Corollary}
\newtheorem{proposition}[theorem]{Proposition}
\title{Converse results, saturation and \quasioptimality for Lavrentiev regularization
of accretive problems}
\numberwithin{equation}{section}
\author{Robert Plato%
\thanks{Department of Mathematics, University of Siegen,
Walter-Flex-Str.~3, 57068 Siegen, Germany.
}
}

\newcommand{\para}{\gamma}
\newcommand{\parab}{\overline{\para}}
\newcommand{\parac}{\para(\delta,\fdelta)}

\newcommand{\paradelta}{\gamma_\delta}
\newcommand{\parazero}{\gamma_0}
\newcommand{\ix}{\mathcal{X}}
\newcommand{\ixh}{\mathcal{H}}
\newcommand{\R}{\mathcal{R}}
\newcommand{\D}{\mathcal{D}}
\newcommand{\N}{\mathcal{N}}
\newcommand{\un}{u_{N}}
\newcommand{\ur}{u_{R}}
\newcommand{\reza}{\mathbb{R}}
\newcommand{\pp}{p}
\newcommand{\qq}{q}
\newcommand{\postype}{nonnegative\xspace}
\newcommand{\ofpostype}{\postype{}\xspace}
\newcommand{\accretive}{accretive\xspace}

\newcommand{\lavmet}{Lavrentiev regularization\xspace}
\newcommand{\lavsmet}{\lavmet{}\xspace}
\newcommand{\upara}[1][\para]{u_{#1}}
\newcommand{\uparab}{\upara[\parab]}
\newcommand{\uparan}{\upara[\para_n]}
\newcommand{\eparaur}[1][\para]{\fparau[#1]{\ur}}

\newcommand{\eparaux}{\upara-u}
\newcommand{\eparauxb}{\uparab-u}
\newcommand{\eparauxn}{\uparan-u}
\newcommand{\eparauy}{\upara \to u}
\newcommand{\fpara}[1][\para]{e_{#1}}
\newcommand{\fparab}{\fpara[\parab]}
\newcommand{\fparau}[2][\para]{e_{#1}(#2)}

\newcommand{\uparadel}[1][\para]{u_{#1}^\delta}
\newcommand{\uparadelc}{\uparadel[\parac]}

\newcommand{\util}{\widetilde{u}}
\newcommand{\fdelta}{f^\delta}

\newcommand{\Ddelta}{\Delta}
\newcommand{\Pdelta}[1][\delta]{P^{#1}}
\newcommand{\Pdeltau}[1][\delta]{P^{#1}(u)}
\newcommand{\Pdeltautil}[1][\delta]{P^{#1}(\util)}
\newcommand{\Rdelta}[2][\delta]{\myR^{#1}_{#2}}
\newcommand{\Rdeltau}[2][\delta]{\myR^{#1}_{#2}(u)}
\newcommand{\Rdeltautil}[2][\delta]{\myR^{#1}_{#2}(\util)}

\newcommand{\Qdelta}[1][\delta]{Q^{#1}}
\newcommand{\Qdeltau}[1][\delta]{Q^{#1}(u)}

\newcommand{\uparafdelb}[1][\fdelta]{u_{\para(\delta,\fdelta)}(#1)}

\newcommand{\Landau}{\mathcal{O}}
\newcommand{\landau}{\mbox{\tiny $ \mathcal{O} $}}
\newcommand{\refeq}[1]{(\ref{eq:#1})}

\newcommand{\nontrivial}{nontrivial\xspace}

\newcommand{\vpara}{v_{\varepsilon}}
\newcommand{\veps}{v_{\varepsilon}}

\newcommand{\phibeta}{\varphi_{\beta}}
\newcommand{\betapara}{\beta(\varepsilon)}
\newcommand{\betaparab}{\beta(\varepsilon)}
\newcommand{\phibetapara}{\varphi_{\betapara}}

\newcommand{\ceps}{c_\varepsilon}
\newcommand{\bn}{\bigskip \noindent}
\newcommand{\rr}{\sigma}
\newcommand{\alp}{\alpha}
\newcommand{\myR}{R}
\newcommand{\deltapara}{\delta}
\newcommand{\orderoptimal}{strongly \quasioptimal{}\xspace}
\newcommand{\orderoptimality}{strong \quasioptimality{}\xspace}

\newcommand{\quasioptimal}{quasi-optimal\xspace}
\newcommand{\Quasioptimal}{Quasi-optimal\xspace}
\newcommand{\quasioptimality}{\quasioptimal{}ity\xspace}
\newcommand{\Quasioptimality}{\Quasioptimal{}ity\xspace}
\newcommand{\pseudooptimal}{weakly \quasioptimal{}\xspace}
\newcommand{\pseudooptimality}{weak \quasioptimality{}\xspace}

\newcommand{\re}{ \textup{Re} \, }
\newcommand{\skp}[2]{\langle #1,  #2 \rangle}
\newcommand{\norm}[1]{\Vert \hspace{0.4mm} #1 \hspace{0.4mm} \Vert}

\newcommand{\normqua}[1]{\norm{#1}^2}
\theoremstyle{nonumberplain}
\theorembodyfont{\normalfont}
\newtheorem{proof}{\normalfont \textsc{Proof}}

\def\endproof{\quad\vbox{\hrule height0.4pt\hbox{%
   \vrule height0.75ex width0.5pt\hskip0.8ex
   \vrule width0.5pt}\hrule height0.4pt
   }}

\newenvironment{myenumerate}{%
\begin{list}{(\alph{enumcount})}
{\setcounter{enumcount}{1}\usecounter{enumcount}
\setlength{\topsep}{1mm}
\setlength{\itemsep}{0mm}
\setlength{\listparindent}{4mm}
\setlength{\parsep}{0mm}
\setlength{\labelwidth}{0mm}
\setlength{\labelsep}{3mm}
\setlength{\itemindent}{3mm}
\setlength{\leftmargin}{0mm}
}}{\end{list}}

\newenvironment{myenumerate_indent}{%
\begin{list}{(\alph{enumcount})}
{\setcounter{enumcount}{1}\usecounter{enumcount}
\setlength{\topsep}{1mm}
\setlength{\itemsep}{-0.5mm}
\setlength{\labelwidth}{3mm}
\setlength{\labelsep}{2mm}
\setlength{\itemindent}{-0.5mm}
\setlength{\leftmargin}{8mm}
}}{\end{list}}

\newenvironment{mylist}{%
\begin{list}{\tinybullet}
{\setlength{\topsep}{0.2cm}
\setlength{\itemsep}{0mm}
\setlength{\labelwidth}{0mm}
\setlength{\labelsep}{3mm}
\setlength{\itemindent}{3mm}
\setlength{\leftmargin}{0mm}
}}{\end{list}}

\newcommand{\tinybullet}{{\tiny \raisebox{0.6mm}{$ \bullet $}}}

\newenvironment{mylist_indent}{%
\begin{list}{\tinybullet}
{\setlength{\topsep}{0.2cm}
\setlength{\itemsep}{-0.5mm}
\setlength{\labelwidth}{2mm}
\setlength{\labelsep}{3mm}
\setlength{\itemindent}{0mm}
\setlength{\leftmargin}{5mm}
}}{\end{list}}

\newcommand{\Deltapara}[1][\para]{\Delta_{#1}^\delta}

\newcommand{\axrassump}{Let $ A: \ix \to \ix $ be a \postype bounded linear operator on a reflexive Banach space $ \ix $}
\newcommand{\axruassump}{Let $ A: \ix \to \ix $ be a \postype bounded linear operator on a reflexive Banach space $ \ix $, and let $ u \in \ix $}
\newcommand{\axassump}{Let $ A: \ix \to \ix $ be a \postype bounded linear operator on a Banach space $ \ix $}
\newcommand{\axuassump}{Let $ A: \ix \to \ix $ be a \postype bounded linear operator on a Banach space $ \ix $, and let $ u \in \ix $}

\newcommand{\ahassump}{Let $ A: \ixh \to \ixh $ be an \accretive bounded linear operator
 on a Hilbert space $ \ixh $}
\newcommand{\ahuassump}{Let $ A: \ixh \to \ixh $ be an \accretive bounded linear operator
 on a Hilbert space $ \ixh $, and let $ u \in \ixh $}

\newcommand{\ahillpassump}{Let $ A: \ixh \to \ixh $ be an \accretive bounded linear operator on a Hilbert space $ \ixh $, with $ \mathcal{R}(A) \not = \ixh $}
\newcommand{\ahuillpassumpzer}{Let $ A: \ixh \to \ixh $ be an \accretive bounded linear operator on a Hilbert space $ \ixh $ with $ \mathcal{R}(A) \not = \ixh $, and let $ 0 \neq u \in \ixh $}
\newcommand{\remarkend}{\quad $ \vartriangle $}
\newcommand{\forexdat}{in case of exact data\xspace}
\newcommand{\fornoisdat}{in case of noisy data\xspace}
\newcommand{\posboundlinear}{\postype bounded linear\xspace}

\begin{document}

\date{}
\maketitle
\newcounter{enumcount}
\renewcommand{\theenumcount}{(\alph{enumcount})}

\bibliographystyle{plain}

\begin{abstract}
This paper deals with \lavsmet for solving linear ill-posed problems, mostly
with respect to \accretive operators on Hilbert spaces. 
We present converse and saturation results which are an important part in regularization theory. As a byproduct we obtain a new result on the \quasioptimality of a posteriori parameter choices.
Results in this paper are formulated in Banach spaces whenever possible.
\end{abstract}
\section{Introduction}
\label{intro}
Converse and saturation results are an important part in regularization theory for solving ill-posed problems. Related results for Tikhonov regularization were developed many years ago and are well known, see Groetsch~\cite[Chapter 3]{Groetsch[84]} and the references therein, or Neubauer~\cite{Neubauer[97]}.
In the present paper we show that similar results can be obtained for \lavsmet 
when \accretive linear bounded, and possibly non-selfadjoint, operators on Hilbert spaces are involved. 
Our work is inspired by the two papers \cite{Neubauer[97],Raus[84]}. At 
several steps, however, the technique used
in the present paper
differs substantially from the one used in the two papers \cite{Neubauer[97],Raus[84]}
since no
spectral decomposition is available in our setting, in general.
As a byproduct we obtain a new result on the optimality of
a posteriori parameter choices for \lavsmet.

We start more generally with the consideration of equations on Banach spaces, i.e.,
\begin{align}
A u = f,
\label{eq:maineq}
\end{align}
where $ A: \ix \to \ix $ is a bounded linear operator
on a real or complex Banach space $ \ix $ with norm $ \norm{\cdot} $, and $ f \in \mathcal{R}(A) $.
Our focus is on operators having a non-closed range $ \mathcal{R}(A) $ which in fact 
implies that the considered equation \refeq{maineq} is ill-posed.
Note, however, that this range condition will be explicitly stated in this paper whenever needed.
In the sequel we restrict the considerations to the following class of
operators:
\begin{definition}
\label{th:postype-def}
A bounded linear operator $ A: \ix \to \ix $ on a Banach space $ \ix $ is called \emph{\ofpostype}, if for any parameter $ \para > 0 $ 
the operator $ A + \para I: \ix \to \ix $ has a bounded inverse on $ \ix $, and
\begin{align}
\norm{(A + \para I)^{-1} } \le \tfrac{ M }{\para}
\quad \text{for } \para > 0,
\label{eq:postype}
\end{align}
holds, with some constant $ M \ge 1 $ that is independent of $ \para $.
\end{definition}
The notation ``\postype'' is introduced by Komatsu~\cite{Komatsu[69.1]}; see also Martinez/Sanz~\cite{Martinez_Sanz[00]}. In many papers, no special notation is used for property \refeq{postype}.
\begin{example}
\label{th:integration_abel_example}
Prominent examples of \postype operators are given by the classical integration operator 
$ (Vu)(x) =\int_{0}^{x} u(y) dy $ 
for $ 0 \le x \le 1 $,
and the Abel integral operators
$ (V^\alp u)(x)
= \frac{1}{\Gamma(\alp)}
\int_{0}^{x}{(x-y)^{-(1-\alp)} u(y) }{dy}
$ 
for $ 0 \le x \le 1 \ (0 < \alp < 1) $. Both operators $ V $ and $ V^\alp $ are considered either on the space
of functions $ \ix = L^p(0,1) $ with $ 1 \le p \le \infty $, or the space of continuous functions $ \ix = C[0,1] $. See, e.g. \cite[Section 1.3]{Plato[95]} for details.
\remarkend
\end{example}
For the regularization of the considered equation $ Au = f $ with 
a \posboundlinear operator $ A $, we consider Lavrentiev's method
\begin{align}
(A + \para I) \uparadel = \fdelta,
\label{eq:lavmet}
\end{align}
where $ \para > 0 $ is a regularization parameter. In addition we have
\begin{align}
\fdelta \in \ix, \quad
\norm{ f - \fdelta } \le \delta,
\label{eq:noisy_data}
\end{align}
where $ \delta > 0 $ is a given noise level.
We next consider fractional powers of the operator $ A $ that may serve as a tool to describe smoothness of solutions for equation \refeq{maineq}.
%
%
%
%
%
%
%
\begin{definition}
\label{th:frac-power}
For $ 0 < \pp < 1 $, the \emph{fractional power} $ A^\pp: \ix \to \ix $ of a \postype bounded linear operator $ A: \ix \to \ix $ on a Banach space $ \ix $ is given by
(see, e.g., Kato~\cite[formula (12)]{Kato[62]})
the improper operator-valued integral
\begin{align}
A^\pp := \frac{\sin \pi \pp }{\pi}
\int_0^\infty s^{\pp-1} (A + sI)^{-1} A \  ds,
\quad 0 < \pp < 1 .
\label{eq:frac-power}
\end{align}
For arbitrary values $ \pp > 0 $, the fractional power
$ A^\pp $ of the operator $ A $ is defined by
$
A^\pp := A^{\pp - \lfloor \pp \rfloor} A^{\lfloor \pp \rfloor},
$
where $ \lfloor \pp \rfloor $ denotes the largest integer which does not exceed $p$.
\remarkend
\end{definition}
%
%
For each $ 0 < \pp < 1 $, the identity \refeq{frac-power} defines a bounded linear operator
$ A^\pp: \ix \to \ix $. 
In inverse problems, smoothness of a solution $ u $ 
of equation \refeq{maineq}
is often described in the form $ u \in \R(A^\pp)$
for some $ \pp > 0 $. This allows to deduce convergence rates
for \lavmet 
in the case of noise-free data
(with respect to $ \para $)
as well as in the case of noisy data 
(in terms of $ \delta $ then).

The subject of this paper is to present converse and saturation results for those convergence rates. In other terms,
the impact of the convergence rates on the smoothness of the solution is considered, and the maximal possible rates are identified.
The outline of the paper is as follows. 
Section \ref{conv_saturation_exdat} deals with
converse and saturation results for \lavmet \forexdat. 
In Section \ref{optimality}, a theorem is presented which provides the basis for 
the converse and saturation results \fornoisdat.
This theorem has also impact on the optimality of parameter choices for \lavmet, and related results are stated in Section \ref{optimality} as a byproduct.
In Sections \ref{conv_noisdat} and \ref{saturation_noisdat},
converse and saturation results \fornoisdat are presented,
and
Section \ref{auxiliary} serves as an appendix which provides some auxiliary results.
The main results of this paper are formulated in Theorems
\ref{th:pseuopt=orderopt},
\ref{th:converse_noisy_data_1} and
\ref{th:saturation_noisy_data}.
%
%
\section{Converse and saturation results \forexdat}
\label{conv_saturation_exdat}
\subsection{Introductory remarks}
Throughout this section let $ A: \ix \to \ix $ be a \postype bounded linear operator on a Banach space $ \ix $. Our main interest are operators with a non-closed range $ \mathcal{R}(A) $, but 
nowhere in this section this is 
explicitly required. 
Comments on the closed range case can be found at the end of this section, cf.~Remark \ref{th:closed_range_exact_data}.

In a first step, and also as preparation for
converse and saturation results related with noisy data presented in the following sections, we consider \lavsmet in case of exact data, and we also introduce the corresponding approximation error:
for any $ u \in \ix $ let $ \upara \in \ix $ and $ \fpara \in \ix $ be given by
\begin{align}
(A+\para I)\upara = f, \qquad
\fparau{u} := \fpara := \upara-u, \quad \para > 0.
\label{eq:upara-epara-def}
\end{align}
We note that the approximation error
$ \fpara $, sometimes also called bias, can be represented as follows:
\begin{align}
\fpara = -\para (A + \para I)^{-1}u \quad \textup{ for } \para > 0.
\label{eq:epara-rep}
\end{align}
The smoothness of a solution $ u $, given in the form $ u \in \R(A^\pp)$ for some $ \pp > 0 $, has impact on the speed of 
convergence $ \upara \to u $ as $ \para \to 0 $.
We cite the following well-known result; for a proof, see, e.g., 
\cite[Example 4.1]{Plato[96]}.
\begin{proposition}
\label{th:direct_exact_data}
\axassump.
If $ u \in \overline{\R(A)} $ then $ \eparauy $ as $ \para \to 0 $.
If moreover $ u \in \R(A^\pp) $ for some $ 0 < \pp \le 1 $
then
$ \norm{\eparaux} = \Landau(\para^\pp) $ as $ \para \to 0 $.
\end{proposition}
For recent results on the convergence of \lavmet with adjoint source conditions $ u \in \R((A^*)^\pp)$ in Hilbert spaces, see, e.g., Hofmann/Kaltenbacher/Resmerita~\cite{Hofmann_Kaltenbacher_Resmerita[16]} and
Plato/Hofmann/Math\'{e}~\cite{Plato_Hofmann_Mathe[16]}.

Two natural questions arise in the context of Proposition \ref{th:direct_exact_data}:
\begin{mylist_indent}
\item
Are the given convergence results in that
proposition
optimal, or, in other terms, are the conditions
$ u \in \overline{\R(A)} $ and
$ u \in \R(A^\pp) $ considered there also necessary, respectively?

\item
Is the range of values for $ \pp $, 
considered in Proposition \ref{th:direct_exact_data},
maximal?
\end{mylist_indent}
We show in this section that the answers to those questions are basically affirmative.
The related results are called converse and saturation results for \lavmet in case of exact data, respectively.

\subsection{Converse results \forexdat}
We start with three converse results related with exact data.
\begin{theorem}
\label{th:converse_exact_data_1}
\axruassump.
If $ \eparauy $ as $ \para \to 0 $,
then necessarily $ u \in \overline{\R(A)} $ holds.
\end{theorem}
\begin{proof}
We consider the decomposition $ u = \ur + \un $ with 
$ \ur \in \overline{\R(A)}, \, \un \in \N(A) $,
see Lemma \ref{th:r-n-decomp} in the appendix.
This decomposition yields
%
\begin{align*}
\para (A + \para I)^{-1}u 
= \para (A + \para I)^{-1}\ur + \un \to \un 
\quad \textup{as } \para \to 0
\end{align*}
according to Proposition \ref{th:direct_exact_data}.
The assumption of the theorem now implies $ \un = 0 $, and
from this the statement of the theorem follows.
\end{proof}
\begin{theorem}
\label{th:converse_exact_data_2}
\axruassump.
If $ \norm{\eparaux} = \Landau(\para^\pp) $ as $ \para \to 0 $ holds for some 
$ 0 < \pp < 1 $,
then we have $ u \in \R(A^\qq) $ for each $ 0 < \qq < \pp $.
\end{theorem}
\begin{proof}
We shall make use of some of the results in Komatsu~\cite[Sections 2 and 3]{Komatsu[66]}. 
The negative fractional power $ A^{-\qq} : \ix \supset \D \to \ix $ is defined in a direct way there, with a domain of definition $ \D $
that, 
under the assumptions on the asymptotical behaviour of $ \norm{\eparaux} $ made in our theorem,
contains $ u $,
cf.~\cite[estimate (3.7)]{Komatsu[66]}. 
We have
$ A^{-\qq}u = \frac{\sin \pi \qq }{\pi}
\int_0^\infty s^{-\qq} (A + sI)^{-1} u \ ds
$ in fact
(see \cite[equation (4.10)]{Komatsu[66]} for the details).
From \cite[Proposition 4.13]{Komatsu[66]} the identity
$ A^{\qq}(A^{-\qq} u) = A^{\qq-\qq} u = u $ then follows 
 which means
$ u \in \R(A^\qq) $, and this completes the proof of the theorem.
\end{proof}

\bn
We note that the statement of Theorem
\ref{th:converse_exact_data_2} cannot be extended to the case $ \qq = \pp $, in general. 
For a counterexample related with Tikhonov regularization, see
Neubauer~\cite[p.~521]{Neubauer[97]}.
In the case $ \pp = 1 $, the situation is different: 
%
\begin{theorem}
\label{th:converse_exact_data_3}
\axruassump. 
In the case $ \norm{\eparaux} = \Landau(\para) $ as $ \para \to 0 $
we necessarily have $ u \in \R(A) $.
\end{theorem}
\proof
From Theorem \ref{th:converse_exact_data_1} we obtain $ u \in \overline{\R(A)} $, and we next show that $ u \in \R(A) $ holds.
By assumption we have $ \norm{(A + \para I)^{-1} u} = \Landau(1) $ as $ \para \to 0 $,
and thus there exists an element $ v \in \ix $ and a sequence $ (\para_n) $ of positive real numbers with $ \para_n \to 0 $ as $ n \to \infty $ such that we have weak convergence
$ (A + \para_n I)^{-1} u \rightharpoonup v $ as $ n \to \infty $.
From this, weak convergence 
$ A (A + \para_n I)^{-1} u \rightharpoonup Av $ as $ n \to \infty $ follows.
On the other hand,
due to
$ u \in \overline{\R(A)} $
we have strong convergence
$ A (A + \para_n I)^{-1} u \to u $ as $ n \to \infty $, and this shows
$ A v = u $. 
\endproof
%

\subsection{Saturation \forexdat}
We next present a saturation result for \lavmet with exact data.
\begin{theorem}
\label{th:saturation_exact_data}
\axruassump.
If $ \norm{\eparaux} = \landau(\para) $ as $ \para \to 0 $,
then necessarily $ u = 0 $ holds.
\end{theorem}
\proof
By assumption we have
$ (A + \para I)^{-1}u \to 0 $ and thus
$ (A + \para I)^{-1} Au \to 0 $ as $ \para \to 0 $.
For the same term there also holds 
$ (A + \para I)^{-1} A u = u - \para (A + \para I)^{-1} u \to u $ as $ \para \to 0 $
according to Proposition \ref{th:direct_exact_data}, and this implies $ u = 0 $.
Note that it follows from Theorem \ref{th:converse_exact_data_1} that $ u \in \overline{\R(A)} $ holds, thus Proposition \ref{th:direct_exact_data} indeed may be applied here.
This completes the proof.
\endproof
\subsection{Some additional observations}
Some conclusions of this section remain true
under weaker hypotheses.
Details are given in the following corollary.
As a preparation we introduce the notation 
$ \reza_+ = \{ \para \in \reza \mid \para > 0 \} $.
\begin{corollary}
\label{th:liminf-weakening-exact-data}
\axruassump.
\begin{myenumerate_indent}
\item
\label{liminf-weakening-exact-data_1}
If there exists some sequence $ (\para_n) \subset \reza_+ $
with
$ 
\lim_{n \to \infty} \para_n = 0 $ such that
$ \norm{\eparauxn} = \Landau(\para_n) $ as $ n \to \infty $ holds,
then we necessarily have $ u \in \R(A) $.

\item
\label{liminf-weakening-exact-data_2}
Suppose that for some sequence $ (\para_n) \subset \reza_+ $
with
$ \lim_{n \to \infty} \para_n =0 $ there holds
\linebreak $ \norm{\eparauxn} = \landau(\para_n) $ as $ n \to \infty $. 
Then we necessarily have $ u = 0 $.

\item
\label{liminf-weakening-exact-data_3}
If $ u \not \in \R(A) $ holds, then 
we have $ \norm{(A + \para I)^{-1}u } \to \infty $
as $ \para \to 0 $.

\item
\label{liminf-weakening-exact-data_4}
If $ u \neq 0 $, then there exists a constant $ c > 0 $ such that
$ \norm{\fpara} \ge c \para $ 
for $ \para > 0 $ small.
\end{myenumerate_indent}
\end{corollary}
\begin{proof}
Parts 
\ref{liminf-weakening-exact-data_1} and \ref{liminf-weakening-exact-data_2}
follow similarly to the proofs of Theorems \ref{th:converse_exact_data_3}
and \ref{th:saturation_exact_data}, respectively;
one has to consider subsequences in those proofs then.
Parts
\ref{liminf-weakening-exact-data_3} and \ref{liminf-weakening-exact-data_4}
are the logical negation of parts 
\ref{liminf-weakening-exact-data_1} and \ref{liminf-weakening-exact-data_2}, respectively.
\end{proof}

\bn
We note that in the case 
``$ \ix $ Hilbert space, $ M = 1 $'' (the operator $ A $ is \accretive then, cf.~the following section),
the modified hypotheses
in parts
\ref{liminf-weakening-exact-data_1} and \ref{liminf-weakening-exact-data_2}
of the preceding corollary
coincide with the original hypotheses considered
in Theorems~\ref{th:converse_exact_data_3} and
\ref{th:saturation_exact_data}, respectively. 
This is an immediate result of
the monotonicity of the functional
$ \para \to \norm{\fpara}/\para $, 
cf.~Lemma~\ref{th:resolvent-monotone}
in the appendix.

We conclude this section with a remark
on the closed range case.
\begin{remark}
\label{th:closed_range_exact_data}
Note that throughout this section we do not require that the range $ \R(A) $ 
is non-closed. In case of a closed range, i.e., $ \R(A) = \overline{\R(A)} $,
the results of
Proposition \ref{th:direct_exact_data}
and Theorems~\ref{th:converse_exact_data_1} and \ref{th:converse_exact_data_3} can be summarized as follows:
\begin{align*}
\lim_{\para \to 0} \upara = u
\ \Longleftrightarrow  \
u \in \R(A)
\ \Longleftrightarrow  \
\norm{\eparaux} = \Landau(\para) \textup{ as } \para \to 0.
\end{align*}
The case $ 0 < p < 1 $ considered in Theorem \ref{th:converse_exact_data_2} is not relevant in the closed range case, while the saturation case considered in Theorem \ref{th:saturation_exact_data} still is.
\remarkend
\end{remark}
%
\section{Optimality concepts}
\label{optimality}
\subsection{Preliminaries}
This section serves on the one hand as a preparation for the noisy data related 
converse and saturation results presented in the subsequent sections.
The results of the present section, however, may be of independent interest: the impact on the optimality of parameter choices for \lavmet is also established.

Our main results in this section are obtained for operators on Hilbert spaces, but the preliminaries presented in this first subsection are considered for Banach spaces.
So, throughout the present subsection we assume that $ A: \ix \to \ix $ is a \postype bounded linear operator on a Banach space $ \ix $. Our main focus is on operators $ A $ with a non-closed range $ \R(A) $ or a nontrivial nullspace $ \N(A) $. 

%
%
%
The \textit{maximal best possible error} of \lavsmet with respect to a given $ u \in \ix $ and a noise level $ \delta > 0 $ is given by
\begin{align}
\Pdeltau  & := 
\sup_{\fdelta: \, \norm{Au-\fdelta} \le \delta }
\ 
\inf_{\para > 0 } \norm{\uparadel - u }
\nonumber
\\
& =
\sup_{\Ddelta \in \ix: \, \norm{\Ddelta} \le \delta }
\ 
\inf_{\para > 0 } \norm{\fpara + (A+\para I)^{-1} \Ddelta }.
\label{eq:pdelta}
\end{align} 
The quantity $ \Pdeltau $ may serve as a tool for considering \quasioptimality of special parameter choices for \lavmet, cf.~Definition \ref{th:optimality-def} below. 
%
%
First, however, we introduce other quantities that are often used in this direction:
%
for $ u \in \ix, \ \delta > 0 $ and $ 1 \le p \le \infty $
we define
\begin{align}
\Rdeltau{p}  := 
\left\{
\begin{array}{lr}
\displaystyle
\inf_{\para > 0 } \big\{ \norm{\eparaux }^p + \big(M\tfrac{\delta}{\para}\big)^p \}^{1/p},
& \textup{if } p < \infty, \\
\displaystyle
\inf_{\para > 0 } \max\big\{ \norm{\eparaux },  M \tfrac{\delta}{\para} \},
& \textup{if } p = \infty,
\end{array}
\right.
\label{eq:rdelta}
\end{align}
where $ M $ is the constant from \refeq{postype}.
Similar to some relations between $ p $-norms on $ \reza^2 $, we obviously have 
$ \Rdeltau{\infty} \le \Rdeltau{p} \le \Rdeltau{1} \le 2\Rdeltau{\infty} $
for each $ u \in \ix $ and $ \delta > 0 $, with $ 1 < p < \infty $.
The most important quantity from this set of numbers is $ \Rdeltau{1} $. This is due to the fact that
$ \norm{\uparadel - u } \le \norm{\eparaux } + M\frac{\delta}{\para} $ holds
for each $ \fdelta \in \ix $ with $ \norm{Au-\fdelta} \le \delta $.

We next introduce two notations related with the optimality of parameter choices;
see Raus/H\"amarik~\cite{Raus_Haemarik[07]} for similar notations.
Other optimality concepts 
can be found in
Vainikko~\cite{Vainikko[87.1]}.
\begin{definition}
\label{th:optimality-def}
\axassump.
We call a parameter choice $ 0 < \para = \para(\delta,\fdelta) \le \infty $
(with the notation $ \uparadel[\infty] := 0 $)
for \lavmet
\begin{mylist_indent}
\item
\emph{\orderoptimal}, if there exists a constant $ c > 0 $ such that
for each $ u \in \ix, \, \delta > 0 $ and $ \fdelta \in \ix $ with $ \norm{Au-\fdelta} \le \delta $,
we have $ \norm{\uparadelc - u } \le c \Pdeltau $,

\item
\emph{\pseudooptimal}, if there exists a constant $ c > 0 $ such that
for each $ u \in \ix, \, \delta > 0 $ and $ \fdelta \in \ix $ with $ \norm{Au-\fdelta} \le \delta $,
we have $ \norm{\uparadelc - u } \le c \Rdeltau{1} $.
\end{mylist_indent}
%
\end{definition}
We obviously have $ \Pdelta(u) \le \Rdeltau{1} $ for each
$ u \in \ix $ and $ \delta > 0 $, so each \orderoptimal parameter choice is \pseudooptimal.

We now consider a modified discrepancy principle (sometimes called MD rule)
which turns out to be a \pseudooptimal
parameter choice strategy 
for \lavmet in Banach spaces.
\begin{example}
\label{th:discrepancy_principle_mod}
Fix  real numbers $ b_1 \ge b_0 > M $, and let 
$ \Deltapara = A\uparadel - \fdelta $ for $ \para > 0 $.
Consider the following parameter choice strategy:
\begin{mylist}
\item
If $ \norm{\fdelta } \le b_1 \delta, $ then take $ \para = \infty $.

\item
Otherwise choose $ 0 < \para = \parac < \infty $ such that
$ b_0 \delta \le  \norm{ \para (A+\para I)^{-1} \Deltapara } \le  b_1 \delta $ holds.
\end{mylist}
%
It is shown in Plato/H\"amarik~\cite[Parameter Choice 4.1 and Theorem 4.4]{Plato_Haemarik[96]} that this parameter choice strategy is \pseudooptimal,
if  $ A : \ix \to \ix $ is a \posboundlinear operator on a Banach space $ \ix $.

It is an open problem, in case of ill-posed problems, if the modified discrepancy principle 
is \orderoptimal in such a general setting. For \accretive operators on Hilbert spaces, however, \orderoptimality can be verified. Details are given in the following subsection.
\remarkend
\end{example}
%
\subsection{\Quasioptimality in Hilbert spaces}
For the following investigations we need to restrict the considered class of operators.
\begin{definition}
\label{th:accretive}
Let $ \ixh $ be a real or complex Hilbert space, with inner product $ \skp{\cdot}{\cdot} $.
A bounded linear operator $ A: \ixh \to \ixh  $
is called \emph{\accretive}, if
\begin{align}
\re \skp{Au}{u} \ge 0 \ \text{for } u \in \ixh.
\label{eq:accretive}
\end{align}
\end{definition}
We note that for real Hilbert spaces,
condition \refeq{accretive} means
$ \skp{Au}{u} \ge 0 $ for each $ u \in \ixh $, and the operator $ A $ is called \emph{monotone} then.
\begin{example}
The classical integration operator and the Abel integral operator 
(see Example~\ref{th:integration_abel_example}), considered on the space
$ L^2(0,1) $ are \accretive.
For the integration operator this follows, e.g., from
Halmos~\cite[Solution 150]{Halmos[78]}, and for the Abel integral operator see, e.g.,
\cite[Theorem 1.3.3]{Plato[95]}.
\end{example}
%
%
A bounded linear operator $ A: \ixh \to \ixh  $ is \accretive if and only it 
satisfies \refeq{postype} with $ M = 1 $. This follows, e.g., from 
Pazy~\cite[Theorem 1.4.2]{Pazy[83]}, in combination with  Lemma~\ref{th:r-n-decomp} in the appendix, applied to the operator $ A + \para I $. 
%

%
%
%
%
%
%
%
Throughout this subsection we consider an \accretive bounded linear operator 
$ A: \ixh \to \ixh $ on a Hilbert space $ \ixh $.
Our main result of this subsection is Theorem \ref{th:pseuopt=orderopt} below, but
first we introduce another quantity which is related to the concept of \pseudooptimality (and variants of it sometimes are used as definition in fact, see, e.g., Hohage/Weidling~\cite{Hohage_Weidling[16]}):
\begin{align}
\Qdeltau  & := 
\inf_{\para > 0 }
\ 
\sup_{\fdelta: \, \norm{Au-\fdelta} \le \delta }
\norm{\uparadel - u }
\nonumber
\\
& =
\inf_{\para > 0 } 
\sup_{\Ddelta \in \ixh: \, \norm{\Ddelta} \le \delta }
\norm{\eparaux + (A+\para I)^{-1} \Ddelta },
\quad u \in \ixh.
\label{eq:qdelta}
\end{align} 
The quantities $\Qdeltau $ and $\Pdeltau $ in \refeq{pdelta} 
differ in such a way that $ \inf $ and $ \sup $ are interchanged.
The following proposition relates $\Qdeltau $ with $ \Rdeltau{p} $ from \refeq{rdelta}, i.e., \pseudooptimality of a parameter choice for \lavmet can by characterized 
by $\Qdeltau $.
Note that in the current situation (accretive operators) we may consider those numbers
$ \Rdeltau{p} $ with $ M = 1 $.
\begin{proposition}
\ahillpassump.
Then the following holds:
\begin{myenumerate_indent}
\item We have $ \Rdelta{2} \le \Qdelta \le \Rdelta{1} $ on $ \ixh $.
\item A parameter choice strategy for \lavmet is \pseudooptimal if and only if
there exists a constant $ c > 0 $ such that
for any $ u \in \ixh, \, \delta > 0 $ and $ \fdelta \in \ixh $ with $ \norm{Au-\fdelta} \le \delta $, we have $ \norm{\uparafdelb - u } \le c \Qdeltau $.
\end{myenumerate_indent}
\end{proposition}
\begin{proof}
The proof of part (a) is elementary, and details are left to the reader. 
We only note that under the given conditions on the operator $ A $ we have
$ \norm{ (A+\para I)^{-1}} = \tfrac{1}{\para} $ for each $ \para > 0 $.
The statement in part (b) is an immediate consequence of part (a). 
\end{proof}

\bn
We note that the assumption $ \R(A) \not = \ixh $ in the preceding proposition is essential.
%


\subsection{Strong versus \pseudooptimality}
%
\ahillpassump.
For the proof of our main theorem of this section,
we need to consider perturbations 
$ \fdelta = f + \delta \veps $ with $ \veps \in \ixh, \norm{\veps} = 1 $, such that 
the data error $ \delta \veps $ is nearly amplified by a factor 
$ 1/\para $ 
when the operator 
$ (A + \para I)^{-1} $ with $ \para > 0 $ is applied to it.
The following lemma provides the basic ingredient.
%
\begin{lemma} 
\label{th:resolvent_maximize}
\ahillpassump.
For parameters $ 0 < \varepsilon \le \para $ and
$ \vpara \in \ixh $ with
$ \norm{\veps} = 1 $ and 
$ \norm{A \veps} = \varepsilon $ we have
\begin{align}
\norm{(A + \para I)^{-1} \veps} 
\ge \big(1 - \frac{\varepsilon}{\para}\big) \frac{1}{\para} .
\label{eq:resolvent_maximize}
\end{align}
\end{lemma}
\proof
We have
\begin{align*}
\norm{ \para (A + \para I)^{-1} \veps } =
\norm{ \veps - A (A + \para I)^{-1} \veps }
\ge 1 - \norm{ (A + \para I)^{-1} A \veps }
\ge 1 - \frac{\varepsilon}{\para},
\end{align*}
and this already completes the proof.
\endproof

\bn
We note that the assumption $ \R(A) \neq \ixh $ in the preceding lemma
is essential. This property is equivalent with $ \R(A) \not = \overline{\R(A)} $ or $ \N(A) \not = \{0\} $,
cf.~Lemma~\ref{th:r-n-decomp} in the appendix. It is also equivalent with
$ 0 \in \sigma(A) $, the spectrum of $ A $. This assumption guarantees,
for arbitrarily small $ \varepsilon > 0 $, the existence
of elements $ \veps \in \ixh $ with the properties stated in Lemma \ref{th:resolvent_maximize}.

\bn
In the proof of Theorem \ref{th:pseuopt=orderopt} considered below, we apply Lemma \ref{th:resolvent_maximize} with some specific $ \veps \in \ixh $ that in fact is obtained by an \accretive transformation of the element $ u $. This guarantees that an inner product that occurs in the mentioned proof takes nonnegative values only.
The following lemma provides the basic ingredient for the construction of those 
special elements $ \vpara $.
%
\begin{lemma}
\label{eq:phibeta-def}
\ahassump.
Let $ u \in \ixh, \, u \not \in \R(A) \cup \N(A) $, and let 
\begin{align}
\phibeta := \frac{(A + \beta I)^{-1} u}{\norm{(A + \beta I)^{-1} u}},
\quad \beta > 0.
\end{align}
For each real number $ 0 < \varepsilon < \frac{\norm{Au}}{\norm{u}} $ there exists a parameter $ \beta = \betapara $ with $ \norm{A\phibetapara} = \varepsilon $.
\label{th:beta_alpha_choice}
\end{lemma}
\proof The function
$ \beta \mapsto \norm{A\phibeta} $
%
obviously is continuous on $ \reza_+ $, and the lemma then follows from 
the asymptotic behaviours
\begin{align} 
& \lim_{\beta \to 0}  \norm{A\phibeta} = 0,
\qquad
\lim_{\beta \to \infty}  \norm{A\phibeta} = \frac{\norm{Au}}{\norm{u}}.
\label{eq:abetaalpha_behaviour}
\end{align} 
In the sequel, the two statements in \refeq{abetaalpha_behaviour} will be verified.
We consider first the case $ \beta \to 0 $.
Obviously $ A (A + \beta I)^{-1} u = u - \beta (A + \beta I)^{-1} u $ is uniformly bounded with respect to $ \beta > 0 $,
and in addition we have $ \norm{(A + \beta I)^{-1} u } \to  \infty $
as $ \beta \to 0 $, cf. 
part \ref{liminf-weakening-exact-data_3} of Corollary \ref{th:liminf-weakening-exact-data}.
This already completes the proof of the first statement in 
\refeq{abetaalpha_behaviour}.

We next consider the case $ \beta \to \infty $.
From a simple expansion and Lemma \ref{th:bias-monotone} in the appendix we obtain
\begin{align*} 
\norm{A\phibeta} 
= \frac{\norm{\beta (A + \beta I)^{-1} Au}}{\norm{\beta(A + \beta I)^{-1} u}}
\to 
\frac{\norm{Au}}{\norm{u}} \quad \textup{as } \beta \to \infty,
\end{align*} 
which is the second statement in
\refeq{abetaalpha_behaviour}. This completes the proof of the lemma.
\endproof
%
%
%
%
%
%

%

\bn
We next show that for \accretive ill-posed operators on Hilbert spaces,
the notions ``\orderoptimal'' and ``\pseudooptimal'' are equivalent.
This theorem provides the main result of this section.
\begin{theorem}
\label{th:pseuopt=orderopt}
\ahillpassump.
Then we have $ \Rdeltau{2} \le \Pdeltau $
for each $ u \in \ixh $ and $ \delta > 0 $.
\end{theorem}
\begin{proof}
1)~In a first part we show that $ \Rdeltau{2} \le \Pdeltau $ holds for
each $ u \not \in \R(A) $. 

\begin{myenumerate}
\item
We consider the trivial case $ u \in \N(A) $ first. 
Then we have $ \fpara = -u $ for each $ \para > 0 $, and in \refeq{pdelta} we may consider 
$ \Delta = - \delta \frac{ u }{\norm{u}} $ then. From this,
$ \Pdeltau \ge \norm{u} = \Rdeltau{2} $ easily follows.

\item
Let us now consider the case
$ u \not \in \R(A) \cup \N(A) $. 
We first show that
\begin{align} 
\Pdeltau^2 \ge
\ 
\inf_{\para > 0 } \Big\{ 
\normqua{\fpara} + \delta^2 \normqua{(A + \para I)^{-1} \veps}
\Big\}
\label{eq:pseuopt=orderopt-a}
\end{align} 
holds, where
$ \veps = -\phibetapara $ for $ 0 < \varepsilon < \frac{\norm{Au}}{\norm{u}} $, 
and $ \phibetapara $ is chosen as in Lemma \ref{th:beta_alpha_choice}.
In fact, we have
%
\begin{align*} 
\veps & =
-\ceps (A+\betaparab I)^{-1}u, 
\quad \textup{with } \ceps = 
\frac{1}{\norm{(A+\betaparab I)^{-1}u}},
\end{align*} 
and then we obviously have
\begin{align} 
\Pdeltau  \ge
\ 
\inf_{\para > 0 } \norm{\fpara + \delta (A+\para I)^{-1} \veps }.
\label{eq:optimal-lemma-a}
\end{align} 
We now expand, for $ \para > 0 $ fixed, 
the term on the right-hand side of
\refeq{optimal-lemma-a}:
%
%
%
\begin{align*} 
& \normqua{\fpara + \delta (A + \para I)^{-1} \veps}
\\
& \qquad 
= \normqua{\fpara} + 2 \delta \re \skp{€\fpara}{(A + \para I)^{-1} \veps}
+ \delta^2 \normqua{(A + \para I)^{-1} \veps}.
\end{align*}
For the inner product we have, by definition,
\begin{align*} 
\skp{€\fpara}{(A + \para I)^{-1} \veps}
= \ceps \para \skp{€(A + \para I)^{-1} u}
{(A+\betaparab I)^{-1} (A + \para I)^{-1} u},
\end{align*}
which has a nonnegative real part since 
the operator $ (A+\betaparab I)^{-1} $ is \accretive.
This implies
\refeq{pseuopt=orderopt-a}.
\item
We next show that the inequality \refeq{pseuopt=orderopt-a} remains valid if the 
infimum on the right-hand side is considered for $ \para $ away from zero. 
For this purpose we choose some
$ \parazero $ with
\begin{align*} 
0 < \parazero < \min \Big\{ \frac{\delta}{2\Pdeltau}, 2\frac{\norm{Au}}{\norm{u}} \Big\}
\end{align*} 
%
and show in the sequel that
\begin{align} 
\Pdeltau^2 \ge
\ 
\inf_{\para \ge \parazero } \Big\{ 
\normqua{\fpara} + \delta^2 \normqua{(A + \para I)^{-1} \veps}
\Big\} 
\quad \textup{ for } \ 0 < \varepsilon \le \frac{\parazero}{2},
\label{eq:pseuopt=orderopt-b}
\end{align} 
holds. In fact, we have
%
\begin{align*} 
\norm{(A + \para I)^{-1} \veps}
& \ge 
\norm{(A + \parazero I)^{-1} \veps}
\ge 
\big(1 - \frac{\varepsilon}{\parazero}\big) \frac{1}{\parazero}
\ge 
\frac{1}{2\parazero} > 
\frac{\Pdeltau}{\delta}
\quad \textup{ for } \
0 < \para \le \parazero,
\end{align*} 
by monotonicity of the norm of the resolvent operator, see Lemma \ref{th:resolvent-monotone} in the appendix, and Lemma \ref{th:resolvent_maximize} has also been applied.
From this, \refeq{pseuopt=orderopt-b} follows easily.

\item
We proceed now with an estimation of the right-hand side of \refeq{pseuopt=orderopt-b}:
For $ \para \ge \parazero $ and $ \varepsilon \le \tfrac{\parazero}{2} $ we have, by
Lemma \ref{th:resolvent_maximize},
\begin{align*} 
\norm{(A + \para I)^{-1} \veps}
\ge 
\big(1 - \frac{\varepsilon}{\para}\big) \frac{1}{\para}
\ge 
\big(1 - \frac{\varepsilon}{\parazero}\big) \frac{1}{\para},
\end{align*} 
and from \refeq{pseuopt=orderopt-b} we then obtain
\begin{align*} 
\Pdeltau^2
& \ge
\inf_{\para \ge \parazero } \Big\{ 
\normqua{\fpara} + \big(1 - \frac{\varepsilon}{\parazero}\big)^2 \big(\frac{\delta}{\para}\big)^2
\Big\}
\ge
\big(1 - \frac{\varepsilon}{\parazero}\big)^2
\inf_{\para \ge \parazero } \Big\{ 
\normqua{\fpara} + \big(\frac{\delta}{\para}\big)^2
\Big\}
\\
& \ge
\big(1 - \frac{\varepsilon}{\parazero}\big)^2
\inf_{\para > 0 } \Big\{ 
\normqua{\fpara} + \big(\frac{\delta}{\para}\big)^2
\Big\}
=
\big(1 - \frac{\varepsilon}{\parazero}\big)^2 \Rdeltau{2}^2
\quad \textup{ for } 0 < \varepsilon \le \frac{\parazero}{2}.
\end{align*} 
Letting $ \varepsilon \to 0 $ now gives
$ \Rdeltau{2} \le  \Pdeltau $, and this completes the first part of the proof.
\end{myenumerate}
2) In the second part of the proof, we show that
the inequality $ \Rdelta{2} \le \Pdelta $ holds not only on $ \ixh\backslash\R(A) $ but all over the Hilbert space $ \ixh $.
\begin{myenumerate}
\item
As a preparation we observe that, for $ \delta > 0 $ fixed, the functionals $ \Pdeltau $ and $ \Rdeltau{2} $ both are continuous with respect to $ u $. In fact, we have
$ \norm{\fparau{u} - \fparau{\util}} \le \norm{u - \util} $ for each 
$ u, \util \in \ixh $, and from this the two inequalities
\begin{align*} 
\vert \Pdeltau -\Pdeltautil \vert & \le \norm{u - \util},
\\
\vert \Rdeltau{2}^2 -\Rdeltautil{2}^2 \vert
& \le
2 \max\{\norm{u},\norm{\util}\} \norm{u - \util},
\quad 
u, \, \util \in \ixh,
\end{align*} 
are easily obtained.

\item
We are now in a position to verify that $ \Rdelta{2} \le \Pdelta $ holds over $ \ixh $. In fact, we already know that this estimate holds 
on $ \ixh\backslash\R(A) $ (see the first part of this proof), and in addition
the functionals $ \Pdelta $ and $ \Rdelta{2} $ are continuous on $ \ixh $ for $ \delta > 0 $ fixed,
see (a) of the second part of this proof.
The assertion now follows from the fact that each \nontrivial linear subspace of a normed space has an empty interior so that any $ u \in \R(A) $ is the limit of a sequence of elements not belonging to $ \R(A) $.
\end{myenumerate}
This completes the proof of the theorem. 
\end{proof}
%


%

\bn
For symmetric, positive semidefinite operators, a result similar to that of Theorem \ref{th:pseuopt=orderopt} can be found in Raus~\cite{Raus[84]}.
As an immediate consequence of Theorem \ref{th:pseuopt=orderopt} we obtain the following result.
\begin{corollary}
\label{th:pseuopt=orderopt-cor}%
\ahillpassump.
Then any parameter choice strategy for \lavsmet 
is \pseudooptimal if and only if it is \orderoptimal.
\end{corollary}
\begin{example}
Under the conditions of Corollary \ref{th:pseuopt=orderopt-cor},
the parameter choice strategy considered in Example \ref{th:discrepancy_principle_mod}
is \pseudooptimal (cf.~again Example \ref{th:discrepancy_principle_mod})
and therefore also \orderoptimal.
For symmetric, positive semidefinite operators, this is already observed in Raus~\cite{Raus[84]}.
\remarkend
\end{example}

\section{Converse results \fornoisdat}
\label{conv_noisdat}
\subsection{Introductory remarks}
The degree of smoothness of a solution $ u $, 
described here by the property $ u \in \R(A^\pp)$ for some $ \pp > 0 $,
has impact on the decay rate of the best possible maximal error
$ \Pdeltau $ as $ \delta \to 0 $.
We cite the following well-known result; for a proof, see, e.g., 
\cite[Example 4.1]{Plato[96]}.
As a preparation we  note that our main interest are operators having a non-closed range $ \mathcal{R}(A) $, but this is nowhere explicitly required in this section.
Further notes on the closed range case are given at the end of this section, cf.~Remark \ref{th:closed_range_noisy_data}.
\begin{proposition}
\label{th:direct_noisy_data}
\axassump.
\begin{myenumerate_indent}
\item
If $ u \in \overline{\R(A)} $ then $ \Pdeltau \to 0 $ as $ \delta \to 0 $.

\item
Let $ 0 < \pp \le 1 $. If $ u \in \R(A^\pp) $ then
$ \Pdeltau  = \Landau(\delta^{\pp/(\pp+1)}) $ as 
$ \delta \to 0 $.
\end{myenumerate_indent}
\end{proposition}
\begin{proof}
The modified discrepancy principle,
cf.~Example \ref{th:discrepancy_principle_mod},
satisfies, see~\cite[Theorems 2.5 and 4.4]{Plato_Haemarik[96]}, $ \uparadel \to u $ as $ \delta \to 0 $ in the case $ u \in \overline{\R(A)} $.
In addition, for $ 0 < \pp \le 1 $ we have
$ \norm{\uparadel-u} = \Landau(\delta^{\pp/(\pp+1)}) $ as 
$ \delta \to 0 $ for each $ u \in \R(A^\pp) $.
The statement of the proposition 
now easily follows.
\end{proof}

\bn
We note that standard a priori parameter choices may be used as well in this proof.
We may address the same topics as for exact data:

\begin{mylist_indent}
\item
Are the given convergence results in Proposition
\ref{th:direct_noisy_data}
optimal, or, in other terms, are the conditions
$ u \in \overline{\R(A)} $ and
$ u \in \R(A^\pp) $ stated in parts (a) and (b) there
also necessary, respectively?

\item
Is the considered range of values for $ \pp $,
considered in part (b) of that proposition,
maximal?
\end{mylist_indent}
We show in this section that the answers to those questions basically are affirmative, when
\accretive operators on Hilbert spaces are considered.
\subsection{The converse results \fornoisdat}
We start with a simple converse result which even holds in reflexive Banach spaces in fact. 
%
%
%
\begin{proposition}
\label{th:converse_noisy_data_2}
\axruassump.
If $ \Pdeltau  \to 0 $ as $ \delta \to 0 $, then necessarily $ u \in \overline{\R(A)} $ holds.
\end{proposition}
\begin{proof}
We consider the decomposition $ u = \ur + \un $ with 
$ \ur \in \overline{\R(A)}, \, \un \in \N(A) $,
see Lemma \ref{th:r-n-decomp} in the appendix.
From this and the consideration of $ \fdelta = A u $ in the definition of $ \Pdelta $
we obtain
\begin{align}
\Pdeltau \ge \inf_{\para>0} \norm{\fpara + 0}
\ge \inf_{\para>0} \norm{-\hspace{-0.5mm}\un + \fparau{\ur}}
\ge \tfrac{1}{M} \norm{\un}
\label{eq:converse_noisy_data_2}
\end{align}
for each $ \delta > 0 $.
The latter estimate in
\refeq{converse_noisy_data_2}
follows again by Lemma \ref{th:r-n-decomp} in the appendix and the fact that
$ \eparaur \in \overline{\R(A)} $ holds for each $ \para > 0 $.
Letting $ \delta \to 0 $ in \refeq{converse_noisy_data_2}
shows $ \un = 0 $ which completes the proof.
\end{proof}
%
%

\bn
The following lemma serves as preparation for the converse and saturation results
related with noisy data.
\begin{lemma}
\label{th:balance_parameter_choice}
\ahuillpassumpzer. Let the parameters
$ \delta > 0 $ and $ \parab > 0 $ be related by
\begin{align} 
\label{eq:balance_parameter_choice}
\deltapara = \parab^2 \norm{(A + \parab I)^{-1} u }.
\end{align}
Then we have
\begin{align} 
\label{eq:saturation_noisy_data_b}
\norm{\eparauxb} = \frac{\deltapara}{\parab} \le \Pdeltau.
\end{align}
\end{lemma}
\begin{proof}
Due to Theorem \ref{th:pseuopt=orderopt} it is sufficient to show that
$ \norm{\eparauxb} = \frac{\deltapara}{\parab} \le \Rdeltau{\infty} $ holds.
In fact, by monotonicity we have
$ \norm{\eparauxb} \le \norm{\eparaux} $ for $ \para \ge \parab $ 
(cf.~Lemma \ref{th:bias-monotone} in the appendix), and
$ \frac{\deltapara}{\parab} \le \frac{\deltapara}{\para} $ evidently holds for
$ 0 < \para \le \parab $.
The identity
in \refeq{saturation_noisy_data_b} is a direct consequence of the identity
\refeq{balance_parameter_choice}.
\end{proof}

\begin{remark}
\label{th:balance_parameter_choice_remark}
In the proofs of the following two theorems, Lemma \ref{th:balance_parameter_choice} is applied by choosing the noise level $ \delta $ as a function of the parameter $ \parab $. This remark, however, considers the converse case where $ \parab = \paradelta $
is chosen as a function of $ \delta > 0 $, i.e.,
\begin{align*} 
\delta =  \paradelta^2 \norm{(A + \paradelta I)^{-1} u }.
\end{align*} 
It immediately follows from Lemma \ref{th:balance_parameter_choice} that this parameter choice strategy is \orderoptimal. 
Note that this strategy is of theoretical interest only, and moreover note that the existence of $ \paradelta $ follows from Corollary~\ref{th:bias-monotone-corollary} in the appendix.
\remarkend
\end{remark}
We now present the main converse result related with noisy data.
\begin{theorem}
\label{th:converse_noisy_data_1}
\ahuassump.
If, for some $ 0 < \pp \le 1 $, we have
 $ \Pdeltau  = \Landau(\delta^{\pp/(\pp+1)}) $ as $ \delta \to 0 $, then $ \norm{\eparaux} = \Landau(\para^\pp) $ as $ \para \to 0 $ holds.
\end{theorem}
\proof 
If $ \R(A) = \ixh $ holds, then the statement of the theorem follows immediately from Proposition~\ref{th:direct_exact_data}.
We now assume that $ \R(A) \not = \ixh $, and without loss of generality we may also assume that $ u \not = 0 $ holds.
Due to the usage of $ \para $ in \refeq{pdelta}, we change notation here
and show $ \norm{\fparab} = \Landau(\parab^\pp) $ as $ \parab \to 0 $.
Let $ \parab > 0 $ be arbitrary but fixed, and let
$ \delta = \delta(\parab) $ be given by
\begin{align*} 
\delta = \parab^2 \norm{(A + \parab I)^{-1} u },
\end{align*} 
cf.~Lemma~\ref{th:balance_parameter_choice}. From that lemma we now obtain
\begin{align*} 
\frac{\delta}{\parab}
\le \Pdeltau \le c \delta^{p/(p+1)} 
\end{align*} 
for some constant $ c $ which may be chosen independently from $ \delta $, and then
$ \delta^{1/(p+1)} \le c \parab $ and thus 
$ \delta^{p/(p+1)}  \le c^\pp \parab^\pp $ holds. This finally gives
\begin{align*} 
\norm{\fparab} 
=
\frac{\delta}{\parab}
\le c \delta^{p/(p+1)} 
\le c^{\pp+1} \parab^\pp,
\end{align*} 
and this completes the proof of the theorem.
\endproof

\bn
As an immediate consequence 
of Theorems \ref{th:converse_exact_data_2},
\ref{th:converse_exact_data_3} and
\ref{th:converse_noisy_data_1} we obtain the following result.
\begin{corollary}
\label{th:converse_noisy_data_1b}
\ahuassump.
\begin{myenumerate_indent}
\item
Let $ 0 < \pp < 1 $. 
If $ \Pdeltau  = \Landau(\delta^{\pp/(\pp+1)}) $ as 
$ \delta \to 0 $, then $ u \in \R(A^\qq) $ for each $ 0 < \qq < \pp $.

\item
If $ \Pdeltau  = \Landau(\delta^{1/2}) $ as 
$ \delta \to 0 $, then $ u \in \R(A) $.
\end{myenumerate_indent}
\end{corollary}
\begin{remark}
\label{th:closed_range_noisy_data}
\axrassump.
Throughout this section we have not required that the range $ \R(A) $ 
is non-closed, in general. In case of a closed range, i.e., $ \R(A) = \overline{\R(A)} $,
the results of
Propositions \ref{th:direct_noisy_data} and \ref{th:converse_noisy_data_2}
can be summarized as follows:
\begin{align}
\lim_{\delta \to 0} \Pdeltau = 0
\ \Longleftrightarrow  \
u \in \R(A)
\ \Longleftrightarrow  \
\Pdeltau  = \Landau(\delta^{1/2}) \textup{ as } 
\delta \to 0.
\label{eq:closed_range_noisy_data_1}
\end{align}
The case $ 0 < p < 1 $ considered in Theorem \ref{th:converse_noisy_data_1} (in the Hilbert space setting in fact) is not relevant in the closed range case.

If we have even $ \R(A) = \overline{\R(A)} $ and $ \N(A) = \{0\} $
(which in fact is equivalent to the identity $ \R(A) = \ix $, 
cf.~Lemma~\ref{th:r-n-decomp} in the appendix),
then
$ \Pdeltau  = \Landau(\delta) $ as $ \delta \to 0 $ holds for each $ u \in \ix $.
This follows from 
$ \max_{\para \ge 0} \norm{(A+\para I)^{-1}} < \infty $.
\remarkend
\end{remark}
We have completed our considerations of converse results for \lavmet \fornoisdat. Saturation will be considered in the next section.
%
\section{Saturation \fornoisdat}
\label{saturation_noisdat}
%
%
%
%
%
%
%
We are now in a position to present a saturation result for \lavmet in case of perturbed data.
\begin{theorem}
\label{th:saturation_noisy_data}
\ahillpassump, and let $ u \in \ixh $.
If $ \Pdelta(u)  = \landau(\delta^{1/2}) $ as $ \delta \to 0 $, then
necessarily $ u  = 0 $ holds.
\end{theorem}
\begin{proof}
We prove the theorem by contradiction and assume that $ u \not = 0 $ holds.
For any $ \parab > 0 $ consider
\begin{align} 
\deltapara = \delta(\parab) :=  \parab^2 \norm{(A + \parab I)^{-1} u } > 0,
\label{eq:balance_parameter_choice_b}
\end{align} 
cf.~Lemma \ref{th:balance_parameter_choice}.
From this lemma we then obtain
\begin{align*} 
\frac{\deltapara}{\parab} \le \Pdeltau
= \landau(\deltapara)^{1/2} \quad \textup{as } \ \parab \to 0,
\end{align*}
and thus 
%
$ \deltapara^{1/2} = \landau(\parab) $ as $\parab \to 0 $.
Note that 
$ \deltapara = \delta(\parab) > 0 $ for each $ \parab > 0 $, and
$ \deltapara \to 0 $ as
$ \parab \to 0 $.
This finally gives
\begin{align*} 
\norm{\eparauxb} 
= \frac{\deltapara}{\parab}
= \landau(\deltapara^{1/2})
= \landau(\parab)
 \quad \textup{ as }
\parab \to 0.
\end{align*}
%
Theorem \ref{th:saturation_exact_data} now yields $ u = 0 $, a contradiction to the assumption made in the beginning of our proof.
\end{proof}
\begin{remark}
We note that the assumptions in 
Theorem \ref{th:saturation_noisy_data}
may be weakened without changing the conclusion of the theorem.
We may in fact replace the condition
$ \Pdelta(u)  = \landau(\delta^{1/2}) $ 
as $ \delta \to 0 $ by
$ \liminf_{\delta \to 0} \Pdelta(u)/\delta^{1/2} = 0 $ there.
The only necessary modification in the proof
of Theorem \ref{th:saturation_noisy_data}
is that $ \parab = \paradelta $ in 
\refeq{balance_parameter_choice_b}
is chosen as a function of $ \delta > 0 $ then, i.e.,
$ \delta =  \paradelta^2 \norm{(A + \paradelta I)^{-1} u } $, and
part \ref{liminf-weakening-exact-data_2} of
Corollary \ref{th:liminf-weakening-exact-data} is also applied
in this case.

Further notes on $ \paradelta $ are given in Remark~\ref{th:balance_parameter_choice_remark}.  Note that we have $ \paradelta \to 0 $ as $ \delta \to 0 $ which 
follows
from Corollary~\ref{th:bias-monotone-corollary} in the appendix.

The weakened version of Theorem \ref{th:saturation_noisy_data}
implies that for given $ u \neq 0 $ and 
$ \delta_0 > 0 $,
there exists a constant $ c > 0 $ such that
\begin{align*}
c \delta^{1/2} \le \Pdeltau \quad  \textup{for } \ 0 < \delta \le \delta_0.
\end{align*}
We note that there exists a result for Tikhonov regularization which is similar to
Theorem~\ref{th:saturation_noisy_data}.
For Tikhonov regularization, however, a weakening like the one considered in the present remark is not possible.
For a counterexample see
Neubauer~\cite{Neubauer[97]}.
\remarkend
\end{remark}
\begin{remark}
\label{th:closed_range_noisy_data_2}
Note that the assumption $ \R(A) \not = \ixh $
 made in Theorem \ref{th:closed_range_noisy_data} 
includes the case $ \overline{\R(A)} = \R(A), \, \N(A) \not = \{0\} $
(closed range, nontrivial nullspace). 
Note moreover that the saturation level is different if $ \R(A) = \ixh $ holds. In this case we have (even for \postype operators on Banach spaces)
$ \Pdeltau  = \Landau(\delta) $ as $ \delta \to 0 $ for each $ u \in \ixh $,
cf.~Remark \ref{th:closed_range_noisy_data}.
\remarkend
\end{remark}
\section{Auxiliary results for \postype operators}
\label{auxiliary}
%
In this section we present some auxiliary results which are being used at several places in this paper.
We start with a structural result on the range and nullspace of a \postype operator on a reflexive Banach space.
\begin{lemma} 
\label{th:r-n-decomp}
For a \posboundlinear operator $ A: \ix \to \ix $ on a reflexive Banach space $ \ix $ we have $ \overline{\R(A)} \oplus \N(A) = \ix $, where the symbol $ \oplus $ denotes direct sum.
In addition, there holds $ \norm{ \un } \le M \norm{ \ur+\un  } $
for each $  \ur \in \overline{\R(A)} $ and each $ \un \in \N(A) $, where 
the constant $ M $ is taken from \refeq{postype}.
\end{lemma}
\begin{proof}
See, e.g., \cite[Theorem 1.1.10]{Plato[95]}.
\end{proof}

\bn
We now present results on the behaviour of
the bias and the resolvent.
%
\begin{lemma}
\label{th:bias-monotone}
\axuassump.
\begin{myenumerate_indent}
\item
The functional
$ \para \mapsto \norm{\fpara} = \norm{\para (A + \para I)^{-1} u } $
%
%
is continuous on $ \reza_+ $.

\item 
We have
$ \lim_{\para\to \infty} \norm{\fpara} = \norm{ u } $.
%
%
\item
If \refeq{postype} holds with $ M = 1 $, 
then $ \para \mapsto \norm{\fpara} $  is monotonically increasing on $ \reza_+ $.
\end{myenumerate_indent}
\end{lemma}
\begin{proof}
Continuity of the mapping $ \para \mapsto \norm{\fpara} $ is obvious. 
The asymptotical behaviour of the bias considered in part (b) follows from the representation
\begin{align} 
\norm{\fpara} = \norm{(\para^{-1} A + I)^{-1} u } = \norm{(I + \rr A)^{-1} u} =: g(\rr)
\quad \textup{with } \rr := \para^{-1}
\label{eq:bias-monotone}
\end{align} 
and by letting $ \sigma \to 0 $ then.

Next we consider monotonicity. For $ 0 \le \rr_1 \le \rr_2 $ we have
$ (I + \rr_2 A)^{-1} (I + \rr_1 A) = \omega I + (1-\omega)(I + \rr_2 A)^{-1} $
with $ 0 \le \omega := \tfrac{\rr_1}{\rr_2} \le 1 $. Therefore
$ \norm{ (I + \rr_2 A)^{-1} (I + \rr_1 A) } \le 1 $ holds, and then
$ \norm{ (I + \rr_2 A)^{-1} u } \le \norm{ (I + \rr_1 A)^{-1} u } $ easily follows.
This means that the functional $ g $ in \refeq{bias-monotone} is monotonically decreasing on $ \reza_+ $, and 
therefore the function $ \norm{\fpara} $ is monotonically increasing with respect to 
$ \para $.
\end{proof}

\bn
As an immediate consequence of Lemma \ref{th:bias-monotone} we obtain the following result.
\begin{corollary}
\axassump.
Then for each $ 0 \not = u \in \ix $, the function
\begin{align*} 
f(\para) = \para^2 \norm{(A + \para I)^{-1} u }, \quad \para > 0,
\end{align*}
is continuous on $ \reza_+ $, and in addition
%
$ \lim_{\para\to 0} f(\para) = 0 $ and
$ \lim_{\para\to \infty} f(\para) = \infty $ holds.
If \refeq{postype} holds with $ M = 1 $,
then the function $ f $  is strictly increasing on $ \reza_+ $.
\label{th:bias-monotone-corollary}
\end{corollary}
In a Hilbert space setting we finally present a monotonicity result for the resolvent.
\begin{lemma}
\label{th:resolvent-monotone}
\ahassump.
For $ u \in \ixh $ fixed, the functional
%
$ r(\para) = \norm{ (A + \para I)^{-1} u }, \para > 0 $,
%
%
%
is monotonically decreasing on $ \reza_+ $.
\end{lemma}
\begin{proof}

\begin{myenumerate}
\item
In a first step we assume that the operator $ A $ has a continuous inverse $ A^{-1}: \ixh \to \ixh $. The functional $ r $ then can be written in the form
$ r(\para) = \norm{(I + \para A^{-1})^{-1} A^{-1}u } $ which according to the proof of 
Lemma \ref{th:bias-monotone} is monotonically decreasing,
since the operator $ A^{-1} $ is \accretive.

\item
We now proceed with the general case for $ A $ and consider the operator 
$ A_\varepsilon = A + \varepsilon I : \ixh \to \ixh $ which obviously is an \accretive invertible operator. The first part of this proof shows that 
$ \para \mapsto \norm{ (A + (\para + \varepsilon) I)^{-1} u } $ is
decreasing on $ \reza_+ $ which means
that $ \para \mapsto \norm{ (A + \para I)^{-1} u } $ is
decreasing on the interval $ (\varepsilon, \infty) $.
Letting $ \varepsilon \to 0 $ then yields the desired monotonicity result.
\end{myenumerate}
This completes the proof of the lemma. 
\end{proof}
%
%

%

\bn
\textbf{Acknowledgment.}
The author 
would like to thank Bernd Hofmann (TU Chemnitz).
Without his continued encouragement and advice this work would not have been possible.

\end{document}